\documentclass[12pt]{amsart}  

\usepackage{amsmath}
\usepackage{amsthm}                 
\usepackage{amssymb}
\usepackage{xypic}
\usepackage{xcolor}
\usepackage{graphicx}
\usepackage{enumitem}

\usepackage{tikz}
\usetikzlibrary{cd}

\usepackage{hyperref}
\hypersetup{
    colorlinks,
    linkcolor={red!50!black},
    citecolor={blue!50!black},
    urlcolor={blue!80!black}
}
\usepackage{aliascnt}[2009/09/08] 

\newcommand{\IP}{{\mathbb{P}}}

\newcommand{\ko}{{\mathcal O}}

\newcommand{\ki}{{\mathcal I}}
\newcommand{\ka}{{\mathcal A}}

\newcommand{\kc}{{\mathcal C}}

\newcommand{\kp}{{\mathcal P}}
\newcommand{\ke}{{\mathcal E}}

\newcommand{\CohE}{{\Coh_\ke(X)}}

\newcommand{\trke}{{\ke}}   

\newcommand{\hd}{\text{hd}}
\newcommand{\pd}{\text{pd}}

\newcommand{\gldim}[1]{\mathrm{gl.dim}(#1)}

\newcommand{\kk}{{\mathbf{k}}}

\newcommand{\sod}[1]{\langle #1\rangle}

\newcommand{\inv}{^{-1}}

\newcommand{\coloneqq}{\mathrel{\mathop:}=}

\newcommand{\ASS}{{\rm(\dag)}}
\DeclareMathOperator{\modules}{-mod}
\newcommand{\mmod}[1]{#1\modules}

\DeclareMathOperator{\Hom}{\mathrm{Hom}}

\DeclareMathOperator{\Ext}{\mathrm{Ext}}
\DeclareMathOperator{\ext}{\mathrm{ext}}

\DeclareMathOperator{\Coh}{\mathrm{Coh}}

\DeclareMathOperator{\add}{\mathrm{add}}

\DeclareMathOperator{\End}{{\mathrm{End}}}

\newcommand{\RRR}{\mathsf{R}}
\newcommand{\DDD}{\mathcal{D}}
\newcommand{\Db}{\DDD^b}

\newcommand{\image}{\text{Im}}

\newcommand{\isom}{ \xrightarrow{ \,\smash{\raisebox{-0.65ex}{\ensuremath{\scriptstyle\sim}}}\,}}

\newcommand{\onto}{\to\!\!\!\!\!\to}
\newcommand{\into}{\hookrightarrow}

\newcommand{\blank}{-}

\newcommand{\univext}[2]{ \genfrac{(}{)}{0pt}{1}{#1}{#2} }
\newcommand{\sunivext}[2]{ \text{\raisebox{0.2ex}{\scalebox{0.8}{$\genfrac{(}{)}{0pt}{1}{#1}{#2}$}}} }

\newcommand{\bib}[6]{{\bibitem{#2} #3: {\emph{#4},} #5#6.}}

\newcommand{\harxiv}[1]{ \href{http://arxiv.org/abs/#1}{\texttt{arXiv:#1}}}
\newcommand{\hyref}[2]{\hyperref[#2]{#1~\ref*{#2}}}

\newtheorem{theorem}{Theorem}[section]

\newtheorem*{theorem*}{Theorem}

  \newaliascnt{proposition}{theorem}
  \newtheorem{proposition}[proposition]{Proposition}
  \aliascntresetthe{proposition}

  \newaliascnt{lemma}{theorem}
  \newtheorem{lemma}[lemma]{Lemma}
  \aliascntresetthe{lemma}

  \newaliascnt{corollary}{theorem}
  \newtheorem{corollary}[corollary]{Corollary}
  \aliascntresetthe{corollary}

  \newtheorem{theoremalpha}{Theorem}

\theoremstyle{definition}

  \newaliascnt{definition}{theorem}
  \newtheorem{definition}[definition]{Definition}
  \aliascntresetthe{definition}

  \newaliascnt{remark}{theorem}
  \newtheorem{remark}[remark]{Remark}
  \aliascntresetthe{remark}

  \newaliascnt{question}{theorem}
  
  \aliascntresetthe{question}

  \newaliascnt{example}{theorem}
  \newtheorem{example}[example]{Example}
  \aliascntresetthe{example}

\begin{document}

\title{Tilting chains of negative curves on rational surfaces}

\author{Lutz Hille}
\author{David Ploog}

\begin{abstract}
We introduce the notion of exact tilting objects, which are partial tilting objects $T$ inducing an equivalence between the abelian category generated by $T$ and the category of modules over the endomorphism algebra of $T$.

Given a chain of sufficiently negative rational curves on a rational surface, we construct an exceptional sequence whose universal extension is an exact tilting object. For a chain of $(-2)$-curves, we obtain an equivalence with modules over a well known algebra.
\end{abstract}

\begingroup\renewcommand\thefootnote{}%
%
\footnote{Contact: \texttt{lhill\_01@uni-muenster.de},
                   \texttt{dploog@math.fu-berlin.de}}
\addtocounter{footnote}{-2}\endgroup

\maketitle

\section*{Introduction} 

\noindent
Tilting objects give rise to equivalences between derived categories but when restricted to the underlying abelian categories, they almost never induce equivalences. In this article, we are interested in equivalences of abelian categories. Therefore, we need to consider partial tilting objects. The aim of this paper is to find conditions when a partial tilting object induces an equivalence of abelian categories. This will be applied to surfaces with chains of negative curves. Before we start with our geometric application, we consider the problem abstractly.

Let $T$ be a partial tilting object in a $\kk$-linear abelian category $\ka$.
Then there is a well-established equivalence of triangulated categories
 $\RRR\Hom(T,\blank) \colon \sod{T} \isom \Db(\mmod{\Lambda})$,
where we write $\Lambda=\End(T)$ for the endomorphism algebra and $\sod{T}$ for the triangulated category generated by $T$ which is closed under summands.
We say that $T$ is \emph{exact tilting} if all surjective morphisms in $\add(T)$ split, see \autoref{def:exact_tilting}.

\begin{theoremalpha}
Let $T$ be an exact tilting object of $\ka$. Then there is an equivalence of abelian categories
 $\Hom(T,\blank) \colon \sod{T}\cap\ka \isom \mmod{\Lambda}$.

Moreover, $\sod{T}\cap\ka$ coincides with the full subcategory of $\ka$ consisting of objects admitting a left resolution by objects of $\add(T)$.
 \end{theoremalpha}

See \autoref{prop:exact_tilt} for the proof. In \autoref{prop:getting_exact_tilt}, we show how exact tilting objects arise as universal extensions of exceptional sequences of objects from $\ka$ with special properties. (See \autoref{sub:universal_extensions} for universal extensions of exceptional sequences with vanishing $\Ext^{>1}$.)

\bigskip

Later we are mainly interested in geometric applications. In fact, for any rational surface there always exists a tilting object \cite{HillePerling}. Starting with a chain of curves, we consider an exceptional sequence adapted to this chain. One expects to understand sheaves in a certain neighbourhood using the corresponding exact tilting objects. For further results on existence and further properties of exceptional sequences on rational surfaces, we refer to \cite{HillePerling1}. For exceptional sequences that are not strong, the algebras can be chosen to be quasi-hereditary. Essentially, this means that the category of filtered modules (with respect to the exceptional sequence) is well understood. We use this property at several places, however, never need the theory of quasi-hereditary algebras in more detail.

Our main example of exact tilting concerns chains of rational curves of negative self-intersection (for short: negative curves) on rational surfaces. More precisely, we study the abelian and triangulated categories generated by ideal sheaves of a chain of negative curves which form a special exceptional sequence. The universal extension of this sequence is an exact tilting bundle. For a more precise statement, see \autoref{thm:abelian_category_geometry}.

\begin{theoremalpha} \label{thma:geometry}
Let $X$ be a smooth, projective surface such that $\ko_X$ is exceptional, and let $C_1,C_2,\ldots,C_t$ be an $A_t$-chain of smooth, rational curves with all $C_i^2\leq-2$. Then $\ke \coloneqq (\ko(-C_1-\cdots-C_t),\ldots,\ko(-C_1),\ko)$ is an exceptional sequence such that its universal extension $T$ is an exact tilting bundle, i.e.\ the associated equivalence of triangulated categories restricts to an equivalence of abelian categories:
\[ \xymatrix@C=8em{
\sod{\ke}            \ar[r]^\cong_{\RRR\Hom(T,\blank)} & \Db(\mmod{\End(T)}) \\
\CohE \coloneqq \trke\cap\Coh(X) \ar[r]^\cong_{\Hom(T,\blank)} \ar@{^{(}-}[u] & \mmod{\End(T)} \ar@{^{(}-}[u]
. } \]
\end{theoremalpha}

\noindent
This result is one technical tool used in \cite{Kalck-Karmazyn} for a Kn\"orrer type category equivalence.

For an exact tilting sheaf, the connection between geometry and representation theory provided by tilting is even stronger than usual. On the negative side, such a strong connection can never work for the category of coherent sheaves itself, since it does not contain projective objects and any equivalence between abelian categories preserves projective objects. Thus we are forced to consider partial tilting sheaves to get an equivalence between abelian categories. On the other hand, this equivalence provides us with projective objects in $\CohE$. Thus, we essentially need to construct sufficiently many projective objects (a projective generator) to get the result.

To illustrate the theorem in a small example, we consider just one smooth, rational curve $C$ on a rational surface $X$. Put $r\coloneqq -(C^2+1)$. For $r \leq 0$, i.e.\ $C^2\geq-1$, the bundle $\ko \oplus \ko(-C)$ is a tilting bundle; it is exact only for $C^2 = -1$. On the other hand, for $r \geq 1$, i.e.\ $C^2\leq -2$, we can consider the universal extension of $\ko(-C)$ by $\ko$; it is
 $0 \to \ko^r \to \univext{\ko(-C)}{\ko^r} \to \ko(-C) \to 0$.
The direct sum $\ko\oplus\univext{\ko(-C)}{\ko^r}$ is an exact tilting bundle for $\sod{\ko(-C),\ko}$. For details, see \autoref{ex:exact_tilting}.

The case of a chain of $(-2)$-curves is of particular interest, since there exist many spherical objects in the subcategory $\CohE$. Those spherical objects induce a braid group action by equivalences of the derived category.
In this particular case, the algebra $\Lambda$ of \hyref{Theorem}{thma:geometry} is well-known in representation theory: it is the Auslander algebra of $\kk[T]/T^{t+1}$. The finite-dimensional algebra $\Lambda$ has previously been studied by several authors, see \cite{BHRR} for references.

Here, we study --- from the geometric point of view --- categories encompassing modules over the Auslander algebra of $\kk[T]/T^{t+1}$.

\subsubsection*{Acknowledgements}
A lot of this research has been conducted at Oberwolfach on a Research-in-pairs stay for which we are very grateful. We also thank Martin Kalck for interest in this work and valuable input.

\section{Exact tilting and adapted exceptional sequences}

\noindent
All varieties, algebras and categories are over a ground field $\kk$ which is assumed to be algebraically closed.

\subsection{Exact tilting objects} \label{sub:exact_tilting}

Let $\ka$ be an abelian category, and $T\in\ka$ be a partial tilting object, i.e.\ $\Ext^{>0}(T,T)=0$ with endomorphism algebra $\Lambda \coloneqq \End(T)$. We write $\sod{T}$ for the triangulated category generated by $T$ (closed under summands) inside $\Db(\ka)$. The category $\sod{T}\cap\ka$ is, in general, additive but not abelian.

Classical tilting theory gives an equivalence of triangulated categories
$\RRR\Hom(T,\blank) \colon \sod{T} \isom \Db(\mmod{\Lambda})$. We introduce the following special property which, roughly saying, states that there are no non-trivial surjections in $T$.

\begin{definition} \label{def:exact_tilting}
A partial tilting object $T\in\ka$ is called \emph{exact tilting} if every surjection between objects in $\add(T)$ splits.
\end{definition}

Recall that $\add(T)$ is the additive category generated by $T$, i.e.\ the subcategory of $\ka$ consisting of finite direct sums of summands of $T$. For a concrete exact tilting object from geometry, see \autoref{sec:chains}.

\begin{lemma} \label{lem:exact_tilting}
A partial tilting object $T$ is exact tilting if and only if $S\onto S'$ implies $\Hom(T_i,S)\onto\Hom(T_i,S')$ for any indecomposable summand $T_i$ of $T$ and $S,S'\in\add(T)$.
\end{lemma}

\begin{proof}
If $T$ is exact tilting, then the surjection $S\onto S'$ admits a section $\sigma\colon S'\to S$. Hence any morphism $f\colon T_i\to S'$ is induced by $\sigma f$.

On the other hand, assume that $T$ satisfies the property of the lemma, and let $S\onto S'$ be a surjection of sums of summands of $T$. If $S'$ is indecomposable, then taking $T_i=S'$ in that property gives the desired splitting right away. If $S'$ is decomposable, then the induced surjections onto direct summands of $S'$, i.e.\ $S\onto S'\onto S'_i$, split and can be combined to a section $S'\to S$.
\end{proof}

\begin{proposition} \label{prop:exact_tilt}
Let $T\in\ka$ be an exact tilting object and $\Lambda = \End(T)$.
Then the equivalence $\Phi = \RRR\Hom(T,\blank) \colon \sod{T} \isom \Db(\mmod{\Lambda})$
restricts to an equivalence of abelian categories $\sod{T}\cap\ka \isom \mmod{\Lambda}$.
\end{proposition}

\begin{proof}
The functor $\Phi$ induces an equivalence between the abelian categories $\mmod{\Lambda}$ and
$\Phi\inv(\mmod{\Lambda})$.
Let $F\in\sod{T}\cap\ka$. We want to show that $\Phi(F)\in\mmod{\Lambda}$. As $F\in\sod{T}$ and $\Hom^i(T,T)=0$ for all $i\neq0$ ($T$ partial tilting), there is an isomorphism $F\cong D^\bullet$, where each component $D^i$ consists of summands from $T$.

By assumption, $D^\bullet$ has a single cohomology object $F$ in degree 0. We now show that $D^\bullet$ can be truncated at 0; thus without loss of generality $D^\bullet$ is a $T$-resolution of $F$. If $D^\bullet=[\cdots\to D^0\to D^1\to\cdots D^a]$ has components in positive degree, then we look at the two rightmost non-zero terms: these form a surjection $s\colon D^{a-1}\onto D^a$ and because $T$ is exact tilting, the induced map $\Hom(D^a,D^{a-1})\to\Hom(D^a,D^a)$ is also surjective. Hence we find a section of $s$ and can split off the subcomplex $D^a\to D^a$ as a direct summand of $D^\bullet$. Iterating this process leaves us with a complex ending in degree 0, hence a resolution of $F$.

Recall that $\Phi(T_i)=P_i$ are the indecomposable projective $\Lambda$-modules. Applying $\Phi$ to $D^\bullet$, we thus get a $P$-resolution of $\Phi(F)$, so that $\Phi(F)$ is a $\Lambda$-module. The resulting functor
 $\Phi \colon \sod{T}\cap\ka \to \mmod{\Lambda}$ is exact as a functor between abelian categories
(i.e.\ no derivation necessary) due to $T$ partial tilting:
  $\RRR\Hom(T,D^\bullet) = \Hom(T,D^\bullet)$.
It is essentially surjective as all projective modules are in the image:
$\Phi(T_i)=P_i$.
\end{proof}

\subsection{Universal extensions} \label{sub:universal_extensions}

Let $\ka$ be a ($\kk$-linear) abelian category with finite-dimensional $\Ext$ groups and $D^b(\ka)$ its bounded derived category, and assume we are given two objects $A,B\in\ka$. Following \cite{HillePerling}, we define the universal (co)extension of $B$ by $A$ by the short exact sequences
\[ \begin{array}{c @{\qquad} r}
 0 \to \Ext^1(A,B)^* \otimes B \to \univext{A}{B^r} \to A \to 0 , & \text{(extension)}   \\[1.0ex]
 0 \to B \to \univext{A^r}{B} \to \Ext^1(A,B) \otimes A \to 0   , & \text{(coextension)}
\end{array} \]
where $r\coloneqq\dim\Ext^1(A,B)$. Both extensions are given by the identity in
 $\End(\Ext^1(A,B)) = \Ext^1(A,B) \otimes \Ext^1(A,B)^*$,
using
\begin{align*}
 \Ext^1(A,B) \otimes \Ext^1(A,B)^* &= \Ext^1( A, \Ext^1(A,B)^*\otimes B) , \\
 \Ext^1(A,B) \otimes \Ext^1(A,B)^* &= \Ext^1( \Ext^1(A,B)\otimes A,B) .
\end{align*}
The notation for the extensions is unambiguous because of universality. The following observations are straightforward computations \cite{HillePerling}:

\begin{lemma} Let $A,B\in\ka$ and $\univext{A}{B^r}$ be their universal extension.
 If $\Ext^1(B,B)=0$, then $\Ext^1(\sunivext{A}{B^r},B)=0$.
 If $\Ext^1(A,A)=\Ext^1(B,A)=\Ext^1(B,B)=0$, then $\Ext^1(\sunivext{A}{B^r},\sunivext{A}{B^r})=0$.

If $(A,B)$ is an exceptional pair with $\Ext^{\geq2}(A,B)=0$, then $B\oplus\univext{A}{B^r}$ is partial tilting.
\end{lemma}

An analogous statement holds for the coextensions, which leads to a partial tilting object $A\oplus\univext{A^r}{B}$ if $(A,B)$ is an exceptional pair with $\Ext^{\geq2}(A,B)=0$.

This process can be iterated to yield the following statement which combines Theorems~4.1 and 5.1 of \cite{HillePerling}. Note that this source provides a slightly more general statement: instead of considering an exceptional sequence of objects of the abelian category $\ka$, one can take them from the derived category $\Db(\ka)$, under the assumption that also negative extensions vanish. Since in our treatment all exceptional objects come from $\ka$, we restrict to $E_i\in\ka$ right away.

\begin{proposition}[\cite{HillePerling}] \label{prop:universal_extension}
Let $(E_1,\ldots,E_t)$ be an exceptional sequence in $\Db(\ka)$ such that $E_i\in\ka$ and $\Ext^{\geq2}(E_i,E_j)=0$ for all $i,j$. Then the object obtained from the sequence via iterated universal (co)extension is partial tilting.
%
%
\end{proposition}

Any exceptional sequence $(E_1,\ldots,E_t)$ gives rise to an equivalence between the triangulated subcategory it generates, $\sod{E_1,\ldots,E_t}$ and the derived category of the endomorphism \emph{dg} algebra of $\bigoplus E_i$; see \cite[Theorem~8.5(c)]{Keller}. However, under the assumptions of \autoref{prop:universal_extension}, via (co)extensions we can avoid the dg algebra and deal with a finite-dimensional algebra instead.

\subsection{Special exceptional sequences and exact tilting objects}
Let $\ka$ be an abelian category and $\ke = (E_1,\ldots,E_t)$ be an exceptional sequence in $\Db(\ka)$. By abuse of notation, we write $\ke$ rather than $\sod{\ke}$ for the triangulated category generated by the sequence.

We consider sequences with the following properties:

\bigskip\noindent
\[
{\ASS} \begin{cases}
  E_i\in\ka               \text{ for } i=1,\ldots,t, \\
  \Ext^{\geq2}(E_i,E_j)=0   \text{ for } i,j=1,\ldots,t, \\
  \dim\Hom(E_i,E_j) \leq 1 \text{ for } i\leq j, \\
  \text{all non-zero maps $E_i\to E_j$ are injective.}
\end{cases}
\]

\begin{proposition} \label{prop:getting_exact_tilt}
Let $(E_1,\ldots,E_t)$ be an exceptional sequence in $\Db(\ka)$ satisfying \ASS, and let $T$ be its universal extension. Then $T$ is an exact tilting object.
\end{proposition}

\begin{proof}
We have to show that there are no non-splitting surjections in $T$. If $\hom(E_i,E_j)=1$ for all $i\leq j$, then the objects $E_i$ form a chain of unique inclusions $E_1\subsetneq E_2\subsetneq\ldots$, and all image objects inside $E_t$ are fixed. In general, they form blocks of trees of such chains, and images in each sink are fixed.

Now we check what happens when going to universal extensions: let $(A,B)$ be an exceptional pair of objects of $\ka$ with $\hom(A,B)=1$, $\ext^1(A,B)=r$ and $\ext^{\geq2}(A,B)=0$. The universal extension of the pair is $B\oplus E$ with $E\coloneqq\sunivext{A}{B^r}$, and from general theory we know
\begin{align*}
 \Hom(E,B) &=\Hom(A,B)     && \implies \image(E\to B) = A\subsetneq B , \\
 \Hom(B,E) &=\Ext^1(A,B)^* && \implies \image(B\to E) \subset \univext{0}{B^r} \subsetneq \univext{A}{B^r} = E.
\end{align*}
Therefore surjections in $T$ only come from identity maps, and hence induce surjections under $\Hom(T_i,\blank)$.
\end{proof}

\begin{proposition} \label{prop:abelian_characterisation}
Let $\ka$ be an abelian category and $\ke = (E_1,\ldots,E_t)$ be an exceptional sequence in $\Db(\ka)$ satisfying \ASS. Let $T$ be the universal extension of $\ke$. Then the following categories are 
equivalent:
\begin{enumerate}[label = (\arabic*)]
\item the additive category $\trke\cap\ka$;
\item the abelian subcategory of $\ka$ generated by $E_1,\ldots,E_t$;
\item the additive subcategory of $\ka$ of objects admitting a filtration with factors
      $E_1,E_1/E_2,\ldots,E_t/E_{t-1}$;
\item $\mmod{\End(T)}$.
\end{enumerate}
Moreover, $\trke \cong D^b(\trke\cap\ka)$ as triangulated categories.
\end{proposition}

\begin{corollary}
In particular, $\trke\cap\ka$ is an abelian category and has the object $T$ as a projective generator.
\end{corollary}

\begin{corollary} \label{cor:global_dimension}
The global dimensions of $\ka$, its subcategory $\trke\cap\ka$ and the algebra $\End(T)$ satisfy
\[ \gldim{\End(T)} = \gldim{\trke\cap\ka} \leq \gldim{\ka} . \]
\end{corollary}

\begin{proof}[Proof of \autoref{cor:global_dimension}]
Comparing Ext groups in the two categories,
\[ \Ext^i_{\trke\cap\ka}(-,-) = \Hom_{\trke}(-,-[i]) = \Hom_{\Db(\ka)}(-,-[i]) = \Ext^i_\ka(-,-) ,\]
shows $\gldim{\trke\cap\ka} \leq \gldim{\ka}$, where the first equality uses $\Db(\trke\cap\ka) = \trke$, and the second relies on $\trke\subseteq\Db(\ka)$ being a full, triangulated subcategory.
The equality of the corollary follows from the equivalence $\mmod{\End(T)}\cong\trke\cap\ka$ of \autoref{prop:abelian_characterisation}.
\end{proof}

\begin{proof}
Write $\kc_{(1)},\kc_{(2)},\kc_{(3)},\kc_{(4)}$ for the four categories of the theorem. We know from \autoref{prop:exact_tilt} that $\kc_{(1)}$ is an abelian category. Obviously, both $\kc_{(1)}$ and $\kc_{(2)}$ contain $E_1,\ldots,E_t$, hence $\kc_{(2)}\subseteq\kc_{(1)}$. On the other hand, $\kc_{(1)}$ is closed under kernels, cokernels and direct sums (all of these are special cases of distinguished triangles), so that $\kc_{(1)}\subseteq\kc_{(2)}$.

For the equivalence of $\kc_{(2)}\cong\kc_{(3)}$, we note that any module over a finite-dimensional algebra has a filtration by simple modules. The statement of (3) is that the objects $E_1,E_2/E_1,\ldots,E_t/E_{t-1}$ are the simples of the abelian category $\kc_{(2)}\cong\mmod{\Lambda}$.

We get $\kc_{(1)}\cong\kc_{(4)}$ from \autoref{prop:exact_tilt} and \autoref{prop:getting_exact_tilt}.

The final statement follows from tilting theory, by \autoref{prop:exact_tilt} we have a commutative diagram whose horizontal arrows are equivalences:
\[ \xymatrix{
\trke        \ar[r]^-\Phi                & \Db(\mmod{\Lambda}) \\
\trke\cap\ka \ar[r]^-\Phi \ar@{^{(}->}[u] & \mmod{\Lambda} \ar@{^{(}->}[u]
 } \]
Hence $\trke \cong \Db(\mmod{\Lambda}) \cong \Db(\trke\cap\ka)$, as claimed.
\end{proof}

\begin{example}
Both propositions fail if the condition $\hom(E_i,E_j)=1$ is removed from \ASS: the full and strong exceptional sequence $(\ko,\ko(1))$ on $\Db(\IP^1)$ satisfies $\ko,\ko(1)\in\Coh(\IP^1)$ and all non-zero morphisms $\ko\to\ko(1)$ are injective. However, the universal extension is just the direct sum $\ko\oplus\ko(1)$, and this bundle is not exact tilting, due to the non-split surjection $\ko^2\onto\ko(1)$.
\end{example}

\section{Chains of negative curves} \label{sec:chains}

\noindent
Let $X$ be a smooth, projective surface. In order to apply the theory of exceptional sequences and tilting, we assume that line bundles on $X$ are exceptional. This property is equivalent to $q(X)=p_g(X)=0$, i.e.\ vanishing irregularity ($q(X)=h^1(\ko_X)=0$) and vanishing geometric genus ($p_g(X)=h^0(\omega_X)=h^2(\ko_X)^*=0$). It holds for rational, e.g.\ toric surfaces, but in fact, in any Kodaira dimension there are surfaces with $q=p_g=0$; see \cite[\S VII.11, VIII.15]{BHPV}. Throughout, we assume:

\medskip\noindent
\emph{$X$ denotes a smooth, projective surface such that $\ko_X$ is exceptional.}
\medskip

From now on we fix a chain $(C_1,\dots,C_t)$ of type $A$ of smooth, rational curves, i.e.\ the curves are pairwise disjoint except that $C_i$ and $C_{i+1}$ intersect transversally in a single point, for $i=1,\ldots,t-1$.

We consider the sequence of line bundles
\[ \ke = (E_0,\dots,E_t) \coloneqq (\ko(-C_1-\dots-C_t),\dots,\ko(-C_1),\ko) . \]

\begin{lemma} \label{lem:geometric_assumption}
The sequence of line bundles $\ke$ is an exceptional sequence.

If $C_i^2\leq-2$ for all $i$, then this sequence satisfies assumption \ASS.
\end{lemma}

\begin{proof}
By our standing assumption that line bundles on $X$ are exceptional, all $E_i$ are exceptional sheaves. Furthermore, for a subchain $D$ of $C_1\cup\dots\cup C_t$, the long cohomology sequence for the short exact sequence $0\to\ko(-D)\to\ko\to\ko_D\to0$ implies $H^*(\ko(-D))=0$. Here we use that all components of $D$ are rational, that $D$ is reduced and connected (hence $H^0(\ko_D)=\kk$), and that $\ko$ is exceptional. For any $i>j$, we have $\Ext^k(E_i,E_j)=H^k(\ko(-D))$ for a divisor $D$ of that type. Hence the sequence is exceptional.

As $\ke$ is a chain of line bundles, all non-zero maps $E_i\to E_j$ are inclusions. The sequence has vanishing $\Ext^2$ for general reasons: for any $i\leq j$, we have $\Ext^2(E_i,E_j)=H^2(\ko_D)$ for a subchain $D$ as above. The short exact sequence $0\to\ko\to\ko(D)\to\ko_D(D)\to0$ induces $H^2(\ko(D))=0$, using that $\ko$ is exceptional and that $\ko_D(D)$ has 1-dimensional support.

We proceed to check $\dim\Hom(E_i,E_j)=1$ for $i\leq j$. This is the place where we use the assumption $C_i^2\leq-2$.
Note that $\Hom(E_i,E_j)=H^0(\ko(D))$ for a subchain $D$ of $C_1\cup\dots\cup C_t$.
If $D$ is irreducible, i.e.\ $j=i+1$, then we get $H^0(\ko(D))=\kk$ from the cohomology sequence of $0\to\ko\to\ko(D)\to\ko_D(D)\to0$ using $\ko_D(D)=\ko_{\IP^1}(m)$ with $m=D^2<0$.
Now, by induction, assume that we know $H^0(\ko(D))=\kk$ for some chain and let $C$ be a curve meeting $D$. We consider the short exact sequence
\[ 0 \to \ko(D) \to \ko(D+C) \to \ko_C(C+D) \to 0 . \]
Since $C$ meets precisely one component of $D$, we have $H^0(\ko_C(C+D) = H^0(\ko_{\IP^1}(C^2+1)) = 0$, using
$C^2+1<0$. Taking global sections thus gives $H^0(\ko(D+C)) = H^0(\ko(D)) = \kk$ by induction.
\end{proof}

\begin{remark}
The proof shows a little more: if just one curve has self-intersection $-1$ and all others satisfy $C_i^2\leq-2$, then \ASS\ still holds.

Even more than one $(-1)$-curve can be supported in certain cases. For instance, it can be checked that a $(-1)(-3)(-1)$-chain satisfies \ASS, but a $(-1)(-2)(-1)$-chain does not. Note that the former chain contracts to a $(-1)$-curve, whereas the latter contracts to a 0-curve.
\end{remark}

\begin{remark}
We now consider the case of two $(-1)$-curves and show that condition \ASS\ fails:
if $C_1$ and $C_2$ are $(-1)$-curves intersecting in a point, then $(\ko(-C_1-C_2),\ko(-C_1),\ko)$ is a strong exceptional sequence with $\dim\Hom(\ko(-C_1-C_2),\ko)=2$.

Blowing down $C_1$ yields $\pi\colon X\to Y$ with a smooth, rational curve $F\subset Y$ such that $F^2=0$. Assume that $F$ is the fibre of a morphism $p\colon Y\to\IP^1$, e.g.\ if $X$ is a Hirzebruch surface. Hence $\ko(C_1+C_2) = \pi^*\ko(F) = \pi^*p^*\ko_{\IP^1}(z)$ for a point $z$ on $\IP^1$. Pulling back the surjection $\ko_{\IP^1}^2\onto\ko_{\IP^1}(z)$ gives $\ko_X^2\onto\ko_X(C_1+C_2)$. Hence the partial tilting bundle
 $T = \ko(-C_1-C_2) \oplus \ko(-C_1) \oplus \ko$ is not exact.

Likewise, it can be shown that the functor $\Hom(T,\blank)$ does not induce an equivalence of abelian categories.
\end{remark}

The exceptional sequence $\ke = (\ko(-C_1-\dots-C_t),\dots,\ko(-C_1),\ko)$ is strong precisely when all $C_i^2\geq-1$. By contrast, we are interested in the case $C_i^2<-1$. One motivation for studying the triangulated category $\trke$ generated by such line bundles is that it contains the torsion sheaves $\ko_{C_1},\ko_{C_2}(-1),\ldots,\ko_{C_t}(-1)$. These are of particular interest when all $C_i^2=-2$, for in that case they form an $A_t$-chain of spherical sheaves and thus give a braid group action on $\Db(X)$. In \cite{IU}, the full subcategory $\Db_{C}(X)$ of $\Db(X)$ of objects supported on $C \coloneqq C_1\cup\ldots\cup C_t$ is studied. The category $\trke$ of this article contains some of the spherical sheaves (one for every irreducible component) but has the advantage of being generated by an exceptional sequence. This allows access to methods from representation theory. We think of $\trke$ as a categorical neighbourhood of the triangulated category generated by $\ko_{C_1},\ko_{C_2}(-1),\ldots,\ko_{C_t}(-1)$.

\begin{example} \label{ex:exact_tilting}
We take up the example from the introduction.
Let $\kp$ and $\ki$ be the partial tilting bundles obtained from universal extension and
coextension, respectively. For $t=1$, these are
\[ \kp = \ko \oplus \univext{\ko(-C)}{\ko^r} , \text{ and }
   \ki = \ko(-C) \oplus \univext{\ko(-C)^r}{\ko} . \]
Of these, $\kp$ is exact tilting but $\ki$ is not --- observe that $\ki$ contains the non-splitting surjections $\sunivext{\ko(-C)^r}{\ko}\onto\ko(-C)$. For $r=1$, i.e.\ a single $(-2)$-curve, the endomorphism algebras are the same: $\End(\kp)=\End(\ki)$.

The object $\kp$ is a projective generator, but $\ki$ is in general not an injective cogenerator (it is an injective cogenerator for the category of $\Delta$-modules).
\end{example}

Next, we spell out what \autoref{lem:geometric_assumption} implies in view of Propositions~\ref{prop:exact_tilt}, \ref{prop:getting_exact_tilt}, \ref{prop:abelian_characterisation}. Note that the sheaves in (3) below are the minimal line bundle and the torsion sheaves supported on the irreducible components of the chain. Therefore, these are the simple objects of the abelian category $\CohE$. Also note that the structure sheaf $\ko$, i.e.\ the maximal line bundle of the sequence is the consecutive extension of these torsion sheaves by $\ko(-C_1-\dots-C_t)$.

\begin{theorem} \label{thm:abelian_category_geometry}
Let $C_1,\ldots,C_t$ be an $A_t$-chain of curves on $X$ such that $C_i\cong\IP^1$ and $C_i^2\leq-2$ for all $i$, let $T$ be the universal extension of the sequence
\[ \ke = (\ko(-C_1-\dots-C_t),\dots,\ko(-C_1),\ko) \]
and let $\Lambda = \End(T)$ be the endomorphism algebra.
Then $T$ is an exact tilting object and induces an equivalence of abelian categories
\[ Hom(T,\blank) \colon \CohE \isom \mmod{\Lambda} . \]
Furthermore, the following categories are equivalent to each other:
\begin{enumerate}[label = (\arabic*)]
\item the additive category $\CohE \coloneqq \trke\cap\Coh(X)$;
\item the abelian subcategory of $\Coh(X)$ generated by the line bundles
      $\ko(-C_1\dots-C_t),\dots,\ko(-C_1),\ko$;
\item the additive category of coherent sheaves admitting a filtration with factors
      $\ko_X(-C_1\cdots-C_t),\ko_{C_t}(-1),\ldots,\ko_{C_2}(-1),\ko_{C_1}$.
\end{enumerate}
Moreover, $\trke \cong \Db(\CohE)$ as triangulated categories and the algebra $\Lambda$ is quasi-hereditary.
\end{theorem}

\begin{proof}
By \autoref{lem:geometric_assumption}, the sequence of line bundles $\ke$ is an exceptional sequence satisfying the condition \ASS. Hence by \autoref{prop:getting_exact_tilt}, the universal extension $T$ of $\ke$ is an exact tilting object and we get the equivalence of abelian categories
 $\RRR\Hom(T,\blank) \colon \CohE \isom \mmod{\Lambda}$
from \autoref{prop:exact_tilt}.

The properties of the abelian category follow from \autoref{prop:abelian_characterisation}, using that the inclusion $\ko(-C_1\ldots-C_{i-1})\into\ko(-C_1-\ldots-C_i)$ has cokernel $\ko_{C_i}(-1)$, for $i>1$.

As $X$ is a smooth surface, $\Coh(X)$ has global dimension 2. Therefore, the category $\CohE$ also has global dimension 2, and hence so has the algebra $\Lambda$. It is a general fact that this already implies $\Lambda$ quasi-hereditary \cite[Theorem~2]{Dlab-Ringel0}.
\end{proof}

\section{First properties of the abelian category $\CohE$}

\begin{lemma} 
\makeatletter
\hyper@anchor{\@currentHref}
\makeatother
\label{lem:sheaf_properties}
\begin{enumerate}
\item A torsion free sheaf in $\CohE$ is locally free. 
\item $\CohE$ is closed under taking torsion subsheaves.
\item $F\in\CohE$ is locally free $\iff$ $\Ext^2(F,\blank)=0$ on $\CohE$.
\item The support of a non-zero object in $\trke$ is either $X$ or a union of curves $C_i$.
\end{enumerate}
\end{lemma}

\begin{proof}
(1) By characterisation (3) of \autoref{thm:abelian_category_geometry}, a sheaf $F\in\CohE$ has a filtration $0=F^0 \subsetneq F^1 \subsetneq \ldots \subsetneq F^l=F$, whose factors $F^i/F^{i-1}$ are either the torsion sheaves $\ko_{C_1}, \ko_{C_2}(-1), \ldots, \ko_{C_t}(-1)$ or the line bundle $\ko(-C_1\ldots-C_t)$. We claim that for $F$ indecomposable this filtration is a refinement of the torsion filtration of $F$: the torsion part of $F$ is the maximal $F^i$ such that all factors up to $F^i$ are torsion.

For this, consider a non-split extension $0 \to M' \to M \to M''\to 0$ of sheaves on $X$ with $M'$ locally free, and $M''$ indecomposable and purely 1-dimensional (i.e.\ supported on a divisor without embedded points). Then the sheaf $M$ is locally free: by assumption, the homological dimensions are $\hd(M'')=1$ and $\hd(M')=0$; as the extension does not split, this implies $\hd(M)=0$.
(Recall the \emph{homological dimension} $\hd(M)=\sup_{x\in X}\pd(M_x)$, the supremum of projective dimensions of stalks of a sheaf $M$. The local situation is $0\to R^r\to M \to R/f \to 0$ for a 2-dimensional regular local ring $R$ and $0\neq f\in R$.)

This also shows (2) and (4), i.e.\ that $\CohE$ is closed under taking torsion subsheaves, and sheaves in $\CohE$ have the supports mentioned in (4). This property immediately extends to objects of $\trke$.

(3) Let $V\in\CohE$ be locally free. Again by \autoref{thm:abelian_category_geometry}, $V$ has a filtration by the line bundles occurring in the exceptional sequence $\ke$. (Note that if $E_i \subset V$, then $V/E_i$ is torsion free, hence locally free again.) Therefore, showing $\Ext^2(V,\blank)=0$ reduces to showing $\Ext^2(E_i,\blank)=0$, but the latter vanishing is clear from the outset.

For the converse, assume $\Ext^2(F,\blank)=0$ and let $0 \to F' \to F \to F'' \to 0$ be the torsion decomposition of $F$, i.e.\ $F'$ is the maximal torsion subsheaf of $F$. For any $A\in\CohE$, we get an exact sequence
  $0 = \Ext^2(F,A) \to \Ext^2(F',A) \to 0$.
Especially for $A=F'$, we obtain $\Ext^2(F',F')=0$. This forces $F'=0$, because $F'$ is filtered by $\ko_{C_1}, \ko_{C_2}(-1),\ldots,\ko_{C_t}(-1)$, and for any smooth, rational curve $C\subset X$ with $C^2<0$, we have
 $\text{ext}^2(\ko_C,\ko_C) = h^1(\ko_C(C)) = -C^2-1$.
\end{proof}

\subsection{Euler pairing and Cartan and matrix} \label{sub:Euler}

Consider the exceptional sequence $\ke = (\ko(-C_1-\cdots-C_t),\ldots,\ko(-C_1),\ko)$, and put $b_i \coloneqq C_i^2+2 \leq 0$.

\begin{lemma} \label{lem:cartan}
The Cartan matrix of $\ke$ is
\[ \left( \begin{array}{*{6}{c}}
1   & b_{t}    &   *      & \cdots &  * & * \\
0   & 1        &  b_{t-1} &        &  * &  * \\
0   & 0        &  1      &        & *  & * \\
\vdots     &  & & \ddots              &   & \vdots \\
0   & 0        &  0      &        &  1 & b_1 \\
0   & 0        &  0      & \cdots &  0 & 1
\end{array} \right) \]
with upper triangular $(i,j)$-entry $c_{ij} \coloneqq b_{t-(i-1)}+b_{t-i}+\cdots+b_{t-(j-2)}$.
Its associated quadratic form is
\[  \sum_{i=1}^t x_i^2 + \sum_{i<j} c_{ij} x_i x_j
  = \sum_{i=1}^t x_i^2 + \sum_{i<j} \sum_{l=i}^{j-1} b_{t-l} x_i x_j \]
\end{lemma}

\begin{proof}
By definition of the Cartan matrix, $c_{ij} = \chi(E_{i+1},E_{j+1})$. The upper triangular shape of the matrix is clear, since $\ke$ is an exceptional sequence. For $i\leq j$, put $D_{ij} \coloneqq C_{t+2-j}+\cdots+C_{t+1-i}$. Riemann--Roch for a rational curve $C$ gives $-C.K_X = 2+C^2$, plugging this into Riemann--Roch for $D_{ij}$ yields
\begin{align*}
 c_{ij} &= \chi(\ko(D_{ij})
         = \tfrac{1}{2}D_{ij}^2 - \tfrac{1}{2}D_{ij}.K_X + \chi(\ko) \\
        &= \tfrac{1}{2}D_{ij}^2 + \tfrac{1}{2} \sum_{l=t+2-j}^{t+1-i}(2+C_l^2) + 1 \\
        &= \sum_{l=t+2-j}^{t+1-i} C_l^2 + (j-i-1) + (j-i) + 1
         = \sum_{l=t+2-j}^{t+1-i} b_l .
\end{align*}
The formula for the quadratic form follows immediately.
\end{proof}

\begin{proposition} \label{prop:pairings}
The Euler pairing is symmetric if and only if all $C_i^2 = -2$.

The quadratic form is positive definite if and only if $C_i^2 = -2$ for all $i$, or if $C_j^2=-3$ for a single curve with $C_i^2=-2$ for the rest.
\end{proposition}

\begin{proof}
The claim about symmetry of the Euler form follows at once from the Cartan matrix computation of \autoref{lem:cartan}. For the second statement, note that two $(-3)$-curves lead to a Cartan minor
\[ \left( \begin{array}{rrr}
   1 & -1 & -1 \\ 0 & 1 & -1 \\ 0 & 0 & 1
\end{array} \right ) \]
whose asociated quadratic form is indefinite. Likewise, a single $(-4)$-curve gives rise to a Cartan minor
 $\bigl(\begin{smallmatrix}1 & -2 \\ 0 & 1 \end{smallmatrix} \bigr)$
 whose quadratic form is negative.
\end{proof}

\subsection{Quivers}

We show the quivers describing $\ke$ and $\Lambda$ in the case of all $C_i^2 = -2$.
The Ext quiver of the exceptional sequence
 $\ke = (E_0,\ldots,E_4) = (\ko(-C_1-\cdots-C_4),\ldots,\ko(-C_1),\ko)$
is
\[ \begin{tikzcd}
  E_0 \ar[r, shift left] \ar[r, shift right, dashed] & 
  E_1 \ar[r, shift left] \ar[r, shift right, dashed] & 
  E_2 \ar[r, shift left] \ar[r, shift right, dashed] & 
  E_3 \ar[r, shift left] \ar[r, shift right, dashed] & 
  E_4 
\end{tikzcd} \]
Straight arrows indicate homomorphisms up to scalars, and dashed arrows 1-extensions. Reducible morphisms (composites) are not shown. Any two compositions of arrows with same source and target commute.

The algebra $\Lambda$ occurs as the endomorphism algebra of the universal extension $T$ of $\ke$. Its indecomposable summands are the projective modules $P(0),\ldots,P(4)$.
Note that $P(0)=E_0=\ko(-C_1-\cdots-C_4)$ is the minimal line bundle.
The quiver of $\Lambda$ is
\[ \begin{tikzcd}
  P(0) \ar[r, shift left, "\alpha"] & 
  P(1) \ar[r, shift left, "\alpha"] \ar[l, shift left, "\beta"] & 
  P(2) \ar[r, shift left, "\alpha"] \ar[l, shift left, "\beta"] & 
  P(3) \ar[r, shift left, "\alpha"] \ar[l, shift left, "\beta"] & 
  P(4)                              \ar[l, shift left, "\beta"]
\end{tikzcd} \]
with a zero relation $\beta\alpha=0$ at $P(0)$, commutativity relations $\alpha\beta=\beta\alpha$ at intermediate vertices $P(1),\ldots,P(3)$ and no relation at $P(4)$.

For arbitrary negative intersection numbers $C_i^2$, the quivers with relations are given in \cite[\S5]{Kalck-Karmazyn}.

\end{document}